\documentclass[12pt]{amsart}

\usepackage{amsmath, amsthm, amssymb, amsfonts, verbatim}
\usepackage[margin=1.1in]{geometry}

\usepackage{mathtools}

\usepackage[pagebackref=true, colorlinks=true, citecolor=blue]{hyperref}

\usepackage{color}

\usepackage[normalem]{ulem}

\def\0{{\bf 0}}

\def\R{{\mathbb R}}
\def\Z{{\mathbb Z}}

\theoremstyle{plain}

\newtheorem{thm}{Theorem}
\newtheorem{cor}[thm]{Corollary}

\newtheorem{lem}[thm]{Lemma}

\theoremstyle{definition}
\newtheorem*{definition}{Definition}

\newtheorem*{theorem*}{Theorem}

\newtheorem{remark}[thm]{Remark}
\newtheorem*{acknowledgement}{Acknowledgements}

\makeatletter
\newcommand{\UWaveDontUseDirectlyOnlyThroughAdd}[2][blue]{\bgroup \markoverwith{\textcolor{#1}{\lower3.5\p@\hbox{\sixly \char58}}}\ULon{#2}}
\makeatother
\newcommand{\SOut}[2][red]{\bgroup\markoverwith {\textcolor{#1}{\rule[.45ex]{2pt}{.1ex}}}\ULon{#2}}

\newcommand{\highlight}[2][yellow]{\bgroup\markoverwith {\textcolor{#1}{\rule[-.2em]{2pt}{1.2em}}}\ULon{#2}}

\newcommand{\Comment}[1]{{\color{blue}#1}}
\newcommand{\ToDo}[1]{\textbf{\Comment{[#1]}}}

\newcommand{\iny}{\ensuremath{\infty}}
\newcommand{\grad}{\ensuremath{\nabla}}
\DeclareMathOperator{\dv}{div} %
\newcommand{\prt}{\ensuremath{\partial}}
\newcommand{\brac}[1]{\ensuremath{\left[ #1 \right]}}
\DeclarePairedDelimiter{\multibrac}{\Big[}{\Big]}
\newcommand{\pr}[1]{\ensuremath{\left( #1 \right) }}
\newcommand{\norm}[1]{\ensuremath{\left\Vert #1 \right\Vert}}

\newcommand{\smallnorm}[1]{\ensuremath{\Vert #1 \Vert}}
\newcommand{\abs}[1]{\left\vert#1\right\vert}
\newcommand{\largeabs}[1]{\bigg\vert#1\bigg\vert}

\DeclarePairedDelimiter{\set}{\{}{\}}
\newcommand{\largeset}[1]{\left\{#1\right\}}

\newcommand{\Ignore}[1]{}

\iffalse
\newcommand{\ScratchWork}[1]
	{\begin{quote}\Comment{\footnotesize
        \medskip

        \noindent#1}
        \end{quote}
	}
\else
\newcommand{\ScratchWork}[1]{}
\fi

\begin{document}

\raggedbottom

\numberwithin{equation}{section}

%
%
\newcommand{\MarginNote}[1]{
    \marginpar{
        \begin{flushleft}
            \footnotesize #1
        \end{flushleft}
        }
    }
%
%
\newcommand{\NoteToSelf}[1]{
    }

%
%
\newcommand{\Obsolete}[1]{
    }

\newcommand{\sg}
    {\ensuremath{\sigma}}    
    
    \newcommand{\be}
    {\ensuremath{\beta}}

   \newcommand{\pib}
    {\ensuremath{g_{\beta}}} 
    
     \newcommand{\sib}
    {\ensuremath{\psi_{\beta}}} 
    
        \newcommand{\besw}
    {\ensuremath{ {\dot{B}}^0_{2,2}(w) }} 
    
           \newcommand{\beswhalf}
    {\ensuremath{ {\dot{B}}^{1 \slash 2}_{2,2}(w) }} 
    
         \newcommand{\besdwhalf}
    {\ensuremath{ {\dot{B}}^{1 \slash 2}_{2,2}(Dw) }} 
       
           \newcommand{\besdw}
    {\ensuremath{ {\dot{B}}^0_{2,2}(Dw) }} 




\title[Euler Equation Effects at a Distance]{Incompressible Euler Equations and the Effect of Changes at a Distance}

\author{Elaine Cozzi}
\address{Department of Mathematics, Oregon State University}
\curraddr{}
\email{cozzie@math.oregonstate.edu}

\author{James P. Kelliher}
\address{Department of Mathematics, University of California, Riverside}
\curraddr{}
\email{kelliher@math.ucr.edu}

 \thanks{}

\date{\today} 



\begin{abstract}
Because pressure is determined globally for the incompressible Euler equations, a localized change to the initial velocity will have an immediate effect throughout space. For solutions to be physically meaningful, one would expect such effects to decrease with distance from the localized change, giving the solutions a type of stability.
Indeed, this is the case for solutions having spatial decay, as can be easily shown. We consider the more difficult case of solutions lacking spatial decay, and show that such stability still holds, albeit in a somewhat weaker form.
\end{abstract}

\maketitle

\section{Introduction and Statement of the Main Theorem}
The Euler equations governing incompressible inviscid fluid flow in $\R^d$ are given by
\begin{align*}
    \begin{matrix}
        (E_d) & \left\{
            \begin{matrix}
                \partial_t u + u \cdot \nabla u = - \nabla p\\
                \text{div } u = 0 \\
                u|_{t = 0} = u^0.
            \end{matrix}
            \right.
    \end{matrix}
\end{align*}
There is extensive literature on existence and uniqueness of solutions to the Euler equations on $\R^d$ in various function spaces.  Short time existence of smooth solutions to ($E_d$) goes back to papers
of Lichtenstein and Gunther \cite{Lichtenstein,Gunther}; \cite{Kato} contains a modern form of these results.  For ($E_2$), surveys on global in time existence of solutions can be found in \cite{Majda} and \cite{Chemin1}.  A classical result for weak solutions of the Euler equations in the plane is due to Yudovich \cite{Yud1}, which gives existence and uniqueness of solutions with velocity in $L^2(\R^2)$ and vorticity
(curl of the velocity) in $L^p\cap L^{\infty}(\R^2)$ for some $p<\infty$.  (Yudovich worked in a bounded domain, but his argument is easily extended to the full plane.)  In recent years, some attention has also been paid to the Euler equations with initial vorticity which does not decay at infinity or may be unbounded (see, for example, \cite{AKLN}, \cite{Cozzi}, \cite{Serfati}, \cite{Taniuchi}, \cite{TTY}, \cite{Vishik}, \cite{Yud2}).  

\Ignore{ 
Here we focus our attention on velocity solutions to the Euler equations in two and three dimensions which are bounded and which have bounded vorticity. The goal of this paper is to establish a stability estimate which measures effects at later times to a far-away change in the initial velocity.  Specifically, we consider two velocity solutions $u_1$ and $u_2$ to ($E_d$) with initial data $u^0_1$ and $u^0_2$, respectively, where $u^0_1$ and $u^0_2$ differ only in a small neighborhood of some point $y_0$.  For this class of solutions, we establish an upper bound on the difference $|u_1 (t,z) - u_2 (t,z)|$ for $t>0$ and for $z$ far away from $y_0$.  As one would expect, our upper bound increases as $t$ increases and decreases as the distance from $z$ to $y_0$ increases.  Moreover, we can show that the upper bound can be improved when we assume decay of the velocity at infinity.  In particular, if we also assume $u_1$ and $u_2$ belong to $L^p(\R^d)$ for some $p<\infty$, then the upper bound on the difference of the two solutions decreases as $p$ decreases.
} 

We consider the question of what happens to the solution to the Euler Equations in $\R^d$ at time $t > 0$ when a change is made to the initial velocity far away from a given point. We focus, in particular, on solutions to the Euler equations in two or three dimensions which have bounded vorticity and bounded velocity, but for which vorticity has no spatial decay. 

For solutions having spatial decay of vorticity, bounds on the effects of far-away changes can be obtained in an elementary manner in the $L^2$-norm of the velocity. We quickly review the  argument, as it helps motivate what follows.

Write $B_R(x)$ for the ball of radius $R > 0$ centered at $x \in \R^d$. Let $\varphi \in C_C^\iny(\R^d)$ be supported in $B_2(0)$ with $\varphi \equiv 1$ on $B_1(0)$, and define $\varphi_R(x) = \varphi(x/R)$. Let $(u_1, p_1)$, $(u_2, p_2)$ be two solutions to the Euler equations,
and let $w = u_1 - u_2$, $p = p_1 - p_2$. Then subtracting $(E_d)$ for $u_2$ from $(E_d)$ for $u_1$, multiplying by $\varphi_R^2 w$, and integrating over time and space, one obtains
\begin{align}\label{phiwId}
	\begin{split}
	\norm{\varphi_R w(t)}_{L^2}^2
		= \norm{\varphi_R w^0}_{L^2}^2 + \int_0^t \int_{\R^d}
			\multibrac{2 (\varphi_R &\grad \varphi_R \cdot u_1) \abs{w}^2
				- 2 (\varphi_R w \cdot \grad u_2) \cdot (\varphi_R w) \\
			&\qquad
				+ 4 p \varphi_R \grad \varphi_R \cdot w}.
	\end{split}
\end{align}
\ScratchWork{
	We have,
	\begin{align*}
		\int (\prt_t w + w \cdot \grad u_2 + u_1 \cdot \grad w + \grad p) \cdot (\varphi_R^2 w)
			= 0
	\end{align*}
	so
	\begin{align*}
		\frac{1}{2} \norm{\varphi_R w(t)}_{L^2}^2
			= I + II + III,
	\end{align*}
	where
	\begin{align*}
		I :&= -\int (w \cdot \grad u_2) \cdot (\varphi_R^2 w)
			= - \int (\varphi_R w \cdot \grad u_2) \cdot (\varphi_R w), \\
		II :&= - \int (u_1 \cdot \grad w) \cdot (\varphi_R^2 w)
			= - \int \varphi_R^2 u_1^i \prt_i w^j w^j
			= - \int \varphi_R^2 u_1 \cdot \grad \abs{w}^2 \\
			&= \int (\dv \varphi_R^2 u_1) \abs{w}^2
			= 2 \int (\varphi_R \grad \varphi_R \cdot u_1) \abs{w}^2, \\
		III :&= - \int (\grad p) \cdot (\varphi_R^2 w)
			= \int p \dv(\varphi_R^2 w)
			= 2 \int p \varphi_R \grad \varphi_R \cdot w.
	\end{align*}
} 

For simplicity, suppose $\grad u_2 \in L^\iny([0, t] \times \R^d)$. Then, assuming that our solutions have sufficient decay to ensure that
\begin{align*}
	A := \norm{\grad u_2}_{L^\iny([0, t] \times \R^d)}
		+ \sum_{j = 1}^2
			\brac{\norm{u_j}_{L^\iny([0, t] \times \R^d)}
			+ \norm{u_j}_{L^\iny(0, t; L^2)}
			+ \norm{p_j}_{L^\iny(0, t; L^2)}
			}
		< \iny,
\end{align*}
we have
\begin{align*}
	\norm{\varphi_R w(t)}_{L^2}^2
		\le \norm{\varphi_R w^0}_{L^2}^2
			+ \int_0^t \brac{\frac{C}{R} A^3
				+ 2A \norm{\varphi_R w(s) }_{L^2}^2
				+ \frac{4A^2}{R}} \, ds.
\end{align*}
Gronwall's lemma then gives the bound,
\begin{align}\label{DecayingBound}
	\norm{\varphi_R w(t)}_{L^2}^2
		\le \brac{\norm{\varphi_R w^0}_{L^2}^2 + \frac{C \max \set{1, A}^3}{R} t}
			e^{2 A t}. 
\end{align}

The bound in (\ref{DecayingBound}) controls how rapidly the values of the initial velocity outside of $B_{2R}(0)$ can affect the value of the velocity inside $B_R(0)$, showing that, for decaying solutions, the Euler equations have a degree of locality.
(Of course, many variations on this theme are possible.)

If we only have bounded, but decaying vorticity, say vorticity lying in $L^\iny([0, t]; L^1 \cap L^\iny)$, then we cannot conclude that $\nabla u_2$ is bounded. Regardless, an argument like that in \cite{Yud2}, employing the Calderon-Zygmund inequality, can be used to handle the integral involving $\grad u_2$ in (\ref{phiwId}). Osgood's lemma is used in place of Gronwall's lemma, leading to a different bound, but we suppress these details in this brief overview.

Mere membership of $u_1$, $u_2$ in $L^\iny([0, t]; L^2)$ gives that $\norm{\varphi_R w(t)}_{L^2} \to 0$ as $R \to \iny$; it does not, however, give a rate of convergence. The rate of convergence in  (\ref{DecayingBound}) comes from bounding the rate of decay of the ``tail'' of the velocity, and through it the pressure.

If we now assume that the velocity and vorticity are bounded but do not decay at infinity, the type of argument that led to (\ref{DecayingBound}) fails completely. The lack of decay introduces several difficulties when estimating terms in (\ref{phiwId}).  First, while it is known that in 2D, $\grad p \in L^\iny([0, t] \times \R^2)$ and $\abs{p(t, x)} \le C(t) \log (\abs{x} + e)$ (\cite{Serfati,Kell}), this information is insufficient to control the pressure term in (\ref{phiwId}). Second, the integral involving $u_1$ in (\ref{phiwId}) can no longer be controlled since $w \notin L^2(\R^d)$. Third, the integral involving $\grad u_2$ cannot be controlled using the approach of Yudovich in \cite{Yud2}, as the vorticity only lies in $L^\iny(\R^d)$, and the Calderon-Zygmund theory does not apply in this setting. Finally, we would expect $\norm{\varphi_R w(t)}_{L^2}^2$ to grow with $R$ anyway; hence, it would appear that such an estimate is unattainable when the solutions do not decay. We should seek instead a weighted $L^\iny$-based estimate.

One can view the bound in (\ref{DecayingBound}) as a type of stability estimate that demonstrates a partial locality of the solutions. Without some such locality, the very physical meaning of the solutions must be called into question: If a small change to the initial velocity at an arbitrarily large distance from a given point can immediately and significantly change the value of the solution at that point, then one could never effectively compute a solution.

The original motivation for this paper was to address this issue of physical meaning for the case of bounded vorticity, bounded velocity (Serfati) solutions to the 2D Euler equations. Such solutions were first derived by Serfati in \cite{Serfati}, where their existence and uniqueness was obtained. Serfati's uniqueness argument easily extends to stability estimates in the $L^\iny$-norm of the velocity,
but that argument cannot be localized. Locality of Serfati solutions has been an unaddressed problem.

Stability estimates in $L^{\infty}$-based norms for Serfati solutions were established in \cite{AKLN, TTY}.
The estimate obtained in \cite{AKLN} is an outgrowth of Serfati's uniqueness proof in \cite{Serfati}, while the estimate in \cite{TTY} is an outgrowth of Vishik's uniqueness proof in \cite{Vishik}, and both \cite{Serfati} and \cite{Vishik} depend critically on $L^\iny$ estimates. (Neither employs an energy argument, nor is the maximum principle a key ingredient.) Our approach, which we describe in more detail below, will be to introduce a weighted $L^\iny$-norm into Vishik's argument, adapting as well some pressure-based estimates in \cite{TTY}.

The weight we use is defined as follows. For fixed $z \in \R^d$ and $\alpha > 0$, define $g_z:\R^d \to \R$ by
\begin{align*}
	g_z(x) = (1 + |x-z|)^{-\alpha}.
\end{align*}
A higher value of $\alpha$ makes the weighted norm more localized. As one might expect, the behavior of the Eulerian pressure limits how large $\alpha$ can be. In particular, the important properties of the Riesz operators described in Theorem \ref{Ward} below play a crucial role in determining this upper limit.

For a velocity field, $u$, on $\R^d$, we define the vorticity, $\omega = \omega(u)$, as
\begin{align*}
	\omega = \grad u - (\grad u)^T.
\end{align*}
(In 2D, one can alternately use $\omega = \prt_1 u^2 - \prt_2 u^1$.)

We now state our main result.     
\begin{thm}\label{main}
For fixed $T>0$, assume that ($u_1, p_1$) and ($u_2, p_2$) satisfy ($E_d$) in $\mathcal{D}^{'}((0,T)\times \R^d)$, with
\begin{equation*}
\begin{split}
	&u_1, u_2 \in L^{\infty}([0,T]; L^p\cap L^{\infty}(\R^d)) \text{ for }p\in[1,\infty], \\
	&\omega_1, \omega_2 \in L^{\infty}([0,T] ; L^{\infty}(\R^d)), \\
	&p_1 = \sum_{k,l=1}^d R_k R_l u_1^k u_1^l, \quad
	p_2 =  \sum_{k,l=1}^d R_k R_l u_2^k u_2^l \text { in } BMO.
\end{split}
\end{equation*} 
Let
\begin{equation*}
	M
		= M(t)
		:=
		\sup_{0\leq s \leq t}\left(1+ \| \omega_1(s)\|_{L^{\infty}}
			+ \| \omega_2(s)\|_{L^{\infty}}
			+ \| u_1(s)\|_{L^{\infty}\cap L^{p}}
			+ \| u_2(s)\|_{L^{\infty}\cap L^{p}} \right),
\end{equation*}
and set $\epsilon_0 = \smallnorm{g_z(u_1^0 - u_2^0)}_{L^\iny}$.  Define the strictly increasing function, $\mu \colon [0, 1] \mapsto [0, 1]$ by $\mu(x) = x (1 - \log x)$. Then, whenever $\alpha<1+d\slash p$,
\begin{align*}
	\norm{g_z (u_1 - u_2)(t)}_{L^\iny}
		\le M
			\mu \pr{
				\min \largeset{
					C \frac{\epsilon_0}{M} + C M t \mu \pr{\pr{\frac{\epsilon_0}{M}}^{e^{-CMt}}},
					1}}.
\end{align*}
The constant $C$ depends only upon $d$ and $\alpha$.
\end{thm}
\begin{remark}
	We can see from Theorem \ref{main} that for any $t \le T$,
	we obtain a useful bound on
	$\norm{g_z (u_1 - u_2)(t)}_{L^\iny}$ as long
	$\| g_z (u_1^0 - u_2^0)\|_{L^{\infty}}$
	is sufficiently small.
	Also, stability in $C^r(\R^d)$ for all $r < 1$ follows from Theorem \ref{main}
	by interpolation.
	\ScratchWork{
		We first need to establish that $u_j \in \dot{C}^\beta(\R^2)$ for all $\beta < 1$.
		For decaying vorticity this follows from the Biot-Savart law combined with
		the Calderon-Zygmund inequality and Sobolev embedding. (Or some of your
		paradifferential magic would do the trick, I imagine.) This result should
		be localizable for non-decaying vorticity using a cutoff function, so I imagine
		it is true. I shall assume that for now.  \ToDo{Yep, you can show that $u_j \in \dot{C}^\beta(\R^2)$ for all $\beta < 1$ in the nondecaying case using LP theory.  Its definitely true.}
		
		Now let $r \in (0, 1)$ and fix any $\beta \in (r, 1)$. Then
		\begin{align*}
			\norm{g_z(u_1 - u_2)}_{C^r}
				= \norm{g_z(u_1 - u_2)}_{L^\iny} + \norm{g_z(u_1 - u_2)}_{\dot{C}^r},
		\end{align*}
		and
		\begin{align*}
			\norm{f}_{\dot{C}^r}
				&= \sup_{x \ne y} \frac{\abs{f(x) - f(y)}}{\abs{x - y}^r}
				= \sup_{x \ne y} \abs{f(x) - f(y)}^{1 - \frac{r}{\beta}}
					\pr{\frac{\abs{f(x) - f(y)}}{\abs{x - y}^\beta}}^{\frac{r}{\beta}} \\
				&\le \norm{f}_{L^\iny}^{1 - \frac{r}{\beta}}
					\norm{f}_{\dot{C}^\beta}^{\frac{r}{\beta}}.
		\end{align*}
		Also,
		\begin{align*}
			\norm{g_z(u_1 - u_2)}_{\dot{C}^\beta}
				\le \norm{g_z}_{C^\beta} \norm{u_1 - u_2}_{C^\beta}
				\le C(M).
		\end{align*}
		The  constant here depends upon $M$, since the bound on the $\dot{C}^\beta$
		norms of $u_1$ and $u_2$ here depends upon $M$.
		
		Hence,
		\begin{align*}
			\norm{g_z(u_1 - u_2)}_{C^r}
				&\le \norm{g_z(u_1 - u_2)}_{L^\iny}
					+ \norm{g_z(u_1 - u_2)}_{L^\iny}^{1 - \frac{r}{\beta}}
					\norm{g_z(u_1 - u_2)}_{\dot{C}^\beta}^{\frac{r}{\beta}} \\
				&\le \norm{g_z(u_1 - u_2)}_{L^\iny}
					+ C(M)^{\frac{r}{\beta}}
					\norm{g_z(u_1 - u_2)}_{L^\iny}^{1 - \frac{r}{\beta}}.
		\end{align*}

	}
\end{remark}

We now briefly describe three examples to which Theorem \ref{main} applies. All of these examples are 2D, for even though Theorem \ref{main} is not dimension-dependent, only in 2D are weak solutions to the Euler equations known to exist.

If $p=\infty$ and $d=2$, then Theorem \ref{main} above addresses Serfati (bounded vorticity, bounded velocity) solutions to the Euler equations in the plane (see \cite{Serfati}), and supplies an alternate proof of uniqueness to those in \cite{Serfati,AKLN,TTY}. Moreover, it is shown in \cite{Gallay} that for Serfati solutions,
\begin{align*}
	\norm{u_j(t)}_{L^\iny}
		\le C \smallnorm{u_j^0}_{L^\iny} \pr{1 + \smallnorm{\omega_j^0}_{L^\iny} t}.
\end{align*}
Hence, $M(t) \le C (M(0) + M(0)^2 t)$ for Serfati solutions.

If $u^0$ belongs to $L^2(\R^2)$ and $\omega^0$ belongs to $L^{\infty}(\R^2)$, then there exists a global-in-time solution to ($E_2$) which conserves energy and for which vorticity remains bounded (see \cite{CozziVV1} for details).  For such solutions, one can apply a simple argument using Littlewood-Paley theory to show that for all $t > 0$, 
\begin{equation*}
\| u (t) \|_{L^{\infty}} \leq C(\| u(t) \|_{L^2} + \| \omega(t) \|_{L^{\infty}})  = C(\| u^0 \|_{L^2} + \| \omega^0 \|_{L^{\infty}}).
\end{equation*}
Thus, Theorem \ref{main} applies for this class of solutions with $p=2$ and $d=2$, and we have $M(t) \le C M(0)$.

If $u^0$ belongs to $L^2(\R^2)$ and $\omega^0$ belongs to $L^p\cap L^{\infty}(\R^2)$ for some $p<\infty$, then Theorem \ref{main} addresses the classical Yudovich solutions to ($E_2$) and we again have $M(t) \le C M(0)$.

Let us consider one simple application of Theorem \ref{main}, say, to be explicit, in the setting of Serfati solutions. Suppose that $u_1^0 = u^2_0$ on $B_R(0)$ for some $R > 0$ and that $u_1^0$ is compactly supported on $B_R(0)$. Then
\begin{align}\label{eps0Bound}
	\epsilon_0
		= \smallnorm{g_0(u_1^0 - u_2^0)}_{L^\iny}
		\le \smallnorm{u_2^0}_{L^\iny(B_R(0)^C} (1 + R)^{-\alpha}.
\end{align}

The bound that results from Theorem \ref{main} at time zero is
\begin{align*}
	\abs{(u^0_1 - u^0_2)(0)}
		\le \norm{g_0 (u^0_1 - u^0_2)}_{L^\iny}
		\le C M \mu(C \epsilon_0/M).
\end{align*}
Because $\mu$ is superlinear ($\lim_{x \to 0^+} \mu'(x) = +\iny$), the bound at time zero is somewhat higher than the initial difference. Nonetheless, because the bound on $\epsilon_0$ in (\ref{eps0Bound}) decreases inversely with $R$, the bound on $\norm{g_0(u_1 - u_2)(t)}_{L^\iny}$ decreases as $R \to \iny$ for any fixed $t > 0$. This allows us to determine the distance, $R$, away from a given point (the origin, in this case) beyond which the initial velocity will not affect the value of the velocity at the given point up to a desired time. 

This example illustrates that the change to the initial velocity in Theorem \ref{main} need not be localized, it need only be far away from the point or points of interest. We see this for spatially decaying solutions as well in (\ref{DecayingBound}), though there the control is stronger and can be more cleanly expressed.

\Ignore{ 
\ToDo{Maybe we should remove the corollary, as the estimate is so messy now. Instead, maybe just comment on some of its implications.}

Now, given $y_0\in\R^d$, assume $u_1^0$ and $u_2^0$ differ only in some neighborhood $B=B_r(y_0)$ of $y_0$.  Then for all $x\in B$,
\begin{equation*}
\begin{split}
& g_z(x) =  \left(\frac{1+ |y_0-z| }{ 1+|x-z|} \right)^{\alpha} g_z(y_0) \leq \left(\frac{1+ |y_0-x| + |x-z| }{ 1+|x-z|} \right)^{\alpha} g_z(y_0)\\
&\qquad  \leq (1+ |y_0-x| )^{\alpha} g_z(y_0)  \leq (1+ r )^{\alpha} g_z(y_0).
\end{split}
\end{equation*}
In this case, we can conclude that 
\begin{equation*}
\|g_z(u_1^0 -u_2^0) \|_{L^{\infty}(\R^d)} \leq (1+ r )^{\alpha} g_z(y_0)\| u_1^0 - u_2^0 \|_{L^{\infty}(B)}.
\end{equation*}
In addition, we have that
\begin{equation*}
|u_1(z) - u_2(z)| = | g_z(z) (u_1(z) - u_2(z))| \leq \| g_z (u_1 - u_2) (t) \|_{L^{\infty}}.
\end{equation*}

\Obsolete{
Now note that by continuity of $g_z$, for any fixed $y_0\in\R^d$ and any $C_1>1$, there exists a sufficiently small neighborhood $B$ of $y_0$ (whose size depends on $y_0$ and $C_1$) such that $\|g_z\|_{L^{\infty}(B)} \leq C_1 g_z(y_0)$.  We can conclude that if $u_1^0(x)$ and $u_2^0(x)$ differ only in $B$, then 
\begin{equation}\label{B}
\|g_z(u_1^0 -u_2^0) \|_{L^{\infty}} \leq C_1 g_z(y_0)\| u_1^0 - u_2^0 \|_{L^{\infty}}.
\end{equation}  
In addition, we have that
\begin{equation*}
|u_1(z) - u_2(z)| = | g_z(z) (u_1(z) - u_2(z))| \leq \| g_z (u_1 - u_2) (t) \|_{L^{\infty}}.
\end{equation*}     
Now assume that for fixed $y_0\in\R^d$, $u_1^0(y_0) \neq u_2^0(y_0)$.  We can use continuity of $g_z$, $u_1^0$ and $u_2^0$ (see, for example, Lemma A.3 of \cite{AKLN}) to conclude that for any $C_0>1$, there exists a sufficiently small neighborhood $B$ of $y_0$ such that 
\begin{equation}\label{B}
\| g_z(u_1^0 - u_2^0) \|_{L^{\infty}(B)} \leq C_0|g_z(y_0)(u_1^0(y_0) - u_2^0(y_0))|.
\end{equation}
In addition, we have that
\begin{equation*}
|u_1(z) - u_2(z)| = | g_z(z) (u_1(z) - u_2(z))| \leq \| g_z (u_1 - u_2) (t) \|_{L^{\infty}}.
\end{equation*}  
Thus we have the following corollary to Theorem \ref{main}.
\begin{cor}\label{maincor}
For fixed $T>0$, let $u_1$, $u_2$, and $\alpha$ be as in Theorem \ref{main}.  Given $y_0\in\R^d$, $C_1>1$, and $B$ satisfying (\ref{B}), assume that $u^0_1 = u^0_2$ on $B^c$.  For any point $z\in\R^d$, with $\|g_z(u_1^0 - u_2^0 )\|_{L^{\infty}}$ sufficiently small, 
\begin{equation*}
|(u_1 - u_2)(t, z)|  \leq M^2C(t+1)\left( \frac{|(u_1^0 - u_2^0 )(y_0)|}{(1+|y_0-z|)^{\alpha}} \right)^{e^{-CMt}}\left(3- \log_2  \left( \frac{|(u_1^0 - u_2^0 )(y_0)|}{(1+|y_0-z|)^{\alpha}}\right) \right)^2.
\end{equation*}  
\end{cor}    }   
Thus we have the following corollary to Theorem \ref{main}.
\ToDo{Change this corollary.}
\begin{cor}\label{maincor}
For fixed $T>0$, let $u_1$, $u_2$, $M$, and $\alpha$ be as in Theorem \ref{main}.  Given $y_0\in\R^d$ and $B=B_r(y_0)$ a neighborhood of $y_0$, assume that $u^0_1 = u^0_2$ on $B^c$.  There exists an absolute constant $C$ and a constant $C_1$ depending on $r$ such that, for any point $z\in\R^d$ with $(1+ r )^{\alpha} g_z(y_0)\| u_1^0 - u_2^0 \|_{L^{\infty}(B)}$ sufficiently small, 
\begin{equation*}
|(u_1 - u_2)(t, z)|  \leq M^2C_1(t+1)\left( \frac{\|u_1^0 - u_2^0 \|_{L^{\infty}(B)}}{(1+|y_0-z|)^{\alpha}} \right)^{e^{-CMt}}\left(3- \log_2  \left( \frac{\| u_1^0 - u_2^0 \|_{L^{\infty}(B)}}{(1+|y_0-z|)^{\alpha}}\right) \right)^2.
\end{equation*}  
\end{cor} 
The estimate in Corollary \ref{maincor} measures far away effects (at later times) of a local change in the initial velocity.
} 

Our proof of Theorem \ref{main} is a localization of the uniqueness proof of Vishik's in \cite{Vishik}, with the pressure, however, handled much as in \cite{TTY}.
We localize Vishik's argument using the function, $g_z$. Each of the terms that appear in \cite{Vishik} have an analog here, though most of these terms must be bounded in a different manner. At the heart of bounding these localized terms are the inequality,
\begin{align}\label{ghIntro}
	g_z(x) h_z(x - y)
		\le (1+|y|)^{\alpha},
\end{align}
established in the proof of Lemma \ref{weight}, and the identity,
\begin{align}\label{gradgz}
	\grad g_z = \psi_z g_z,
\end{align}
where $\psi_z$ is a vector-valued function with $L^\iny$-norm independent of $z \in \R^d$ (in fact, $\norm{\psi_z}_{L^\iny} = \alpha$).

The identity in (\ref{ghIntro}) is used several times, in each case to bound terms of the form $g_z(x) (\phi * f)(x)$, as follows:
\begin{align*}
	g_z(x) &(\phi * f)(x)
		= \int \phi(y) g_z(x - y) g_z(x)  h_z(x - y) f(x - y) \, dy \\
		&\le \int \abs{\phi(y)} (1 + \abs{y})^\alpha \abs{(g_z f) (x - y)} \, dy
		= (\abs{\phi(\cdot)}(1 + \abs{\cdot})^\alpha * \abs{g_z f})(x).
\end{align*}
When $\phi$ has sufficient decay, this effectively brings $g_z$ into the integral to obtain an estimate based on the function $g_z f$.

The identity in (\ref{gradgz}) is used to control the one term that is the analog of the integral involving $u_1$ in (\ref{phiwId}). As a consequence of (\ref{gradgz}), $\abs{\grad g_z} \le C \abs{g_z}$. Such a bound is never possible for a compactly supported function: this inability is the fundamental reason why we cannot use the ``cleaner'' localization that a compactly supported function would provide. This same issue shows up, for instance, in \cite{Gallay} (see Section 3.3), though it is addressed in a different manner.

In \cite{Vishik}, decay of vorticity is assumed, which allows the pressure to be easily dealt with. Without such decay, the pressure requires much more careful treatment. We follow the approach in \cite{TTY}, localizing it, much as we localize the approach in \cite{Vishik}.

\begin{remark}
Being a localization of the approaches in \cite{Vishik,TTY}, we expect that a result similar to Theorem \ref{main} will hold for vorticity in the spaces containing unbounded functions
($bmo(\R^d)$, for example)
considered in \cite{Vishik,TTY}. However, for simplicity, we do not consider unbounded vorticity here.  
\end{remark}

Finally, we mention that the proof of uniqueness of Serfati solutions in \cite{Serfati} does not rely upon paradifferential calculus, and hence is not so tightly tied to the full plane. Indeed, Serfati's approach was extended to a 2D exterior domain in \cite{AKLN}. The proof involves bounding the difference between the flow maps for the two solutions (in this, it has similarities to the proof in \cite{MarchioroPulvirenti} of uniqueness for Euler solutions in a 2D bounded domain). This approach seems resistant, however, to localization. It thus remains an interesting open problem whether an analog of Theorem \ref{main} can be obtained for bounded vorticity, bounded velocity solutions in a 2D exterior domain.

\Obsolete{\begin{remark}
Theorem \ref{main} and Corollary \ref{maincor} do not address solutions to the Navier-Stokes equations ($NS$).  However, an inviscid limit result from \cite{CozziVV} approximates solutions of ($NS$) uniformly by solutions to the Euler equations, and can be used with Theorem \ref{main} to establish an estimate for solutions to ($NS$).  Specifically, for $u_{\nu}$ a solution to ($NS$) with $u_{\nu}$ and $\omega_{\nu}$ in $L^{\infty}([0,T];L^{\infty}(\R^d))$, and $u$ as in Theorem \ref{main}, both with the same initial data $u^0$, we have, for sufficiently small $\nu$,  
\begin{equation*} 
\| u_{\nu} - u \|_{L^{\infty}([0,T]; L^{\infty}(\R^2))} \leq C_{1}(2- \log \nu ) \nu^{Ce^{-C_{1}}}. 
\end{equation*}
(The constant $C_1$ depends on $T$ and on the $L^{\infty}$-norms of the Navier-Stokes and Euler velocities and vorticities.  We refer the reader to \cite{CozziVV} for further details.)  This estimate, combined with Corollary \ref{maincor}, allow us to conclude that for two solutions $u_{\nu,1}$ and $u_{\nu,2}$ to ($NS$) with initial data $u^0_1$ and $u^0_2$, respectively, 
\begin{equation*}
\begin{split}
&| (u_{\nu,1} - u_{\nu,2})(t,z)| \leq C_{1}(2- \log \nu ) \nu^{Ce^{-C_{1}}} + \\
&\qquad M^2C(t+1)\left(  \frac{|(u_1^0 - u_2^0)(y_0)|}{(1+|y_0-z|)^{\alpha}} \right)^{e^{-CMt}}\left(3- \log_2  \left(  \frac{|(u_1^0 - u_2^0)(y_0)|}{(1+|y_0-z|)^{\alpha}}\right) \right)^2.
\end{split}
\end{equation*}  
\end{remark}    }   
\section{Definitions and Preliminary Lemmas}
We first define the Littlewood-Paley operators.  It is classical that there exists two functions ${\chi}, {\phi} \in S(\R^d)$ with supp $\hat{\chi}\subset \{\xi\in \R^d: |\xi |\leq \frac{5}{6} \}$ and supp $\hat{\phi}\subset \{\xi\in \R^d: \frac{3}{5} \leq|\xi |\leq \frac{5}{3} \}$, such that, if for every $j\geq 0$ we set $\phi_j(x) = 2^{jd} \phi(2^j x)$ then
\begin{equation*}
\begin{split}
	&\hat{\chi}+ \sum_{j\geq 0} \hat{\phi_j}
		= \hat{\chi} + \sum_{j\geq 0} \hat{\phi}(2^{-j} \cdot) 
		\equiv 1.
\end{split}
\end{equation*}

For $n\in\Z$ define ${\chi}_n \in S(\R^d)$ in terms of its Fourier transform ${\hat{\chi}}_n$, where ${\hat{\chi}}_n$ satisfies 
\begin{equation*}
{\hat{\chi}}_n (\xi) =   \hat{\chi}(\xi) + \sum_{j\leq n} \hat{\phi_j}(\xi)
\end{equation*}
for all $\xi\in\R^d$.  For $f\in S'(\R^d)$ define the operator $S_n$ by  
\begin{equation*}
S_n f = {{\chi}}_n \ast f.
\end{equation*}
Finally, for $f\in S'(\R^d)$ and $j\geq -1$, define the Littlewood-Paley operators ${\Delta}_j$ by
\begin{align*}
    \begin{matrix}
        &\Delta_j f  = \left\{
            \begin{matrix}
                 \chi\ast f,  \qquad j=-1\\
                \phi_j\ast f, \qquad j\geq 0.
            \end{matrix}
            \right.
    \end{matrix}
\end{align*}

We will also need the paraproduct decomposition introduced by J.-M. Bony in \cite{Bony}.  We recall the definition of the paraproduct and remainder used in this decomposition. 
\begin{definition}\label{para}
Define the paraproduct of two functions $f$ and $g$ by
\begin{equation*}
T_fg = \sum_{\stackrel{i,j}{i\leq j-2}} \Delta_i f\Delta_j g = \sum_{j=1}^{\infty} S_{j-2}f\Delta_j g.
\end{equation*}
We use $R(f,g)$ to denote the remainder.  $R(f,g)$ is given by the following bilinear operator:  
\begin{equation*}
R(f,g)=\sum_{\stackrel{i,j}{|i-j|\leq 1}} \Delta_if\Delta_jg.
\end{equation*}
\end{definition}
\noindent Bony's decomposition then gives
\begin{equation*}
fg=T_fg + T_gf + R(f,g).
\end{equation*} 

We will make use of Bernstein's Lemma in what follows.  A proof of the lemma can be found in \cite{Chemin}, chapter 2.
\begin{lem}\label{bernstein}
(Bernstein's Lemma) Let $r_1$ and $r_2$ satisfy $0<r_1<r_2<\infty$, and let $p$ and $q$ satisfy $1\leq p \leq q \leq \infty$. There exists a positive constant $C$ such that for every integer $k$ , if $u$ belongs to $L^p(\R^d)$, and supp $\hat{u}\subset B(0,r_1\lambda)$, then 
\begin{equation}\label{bern1}
\sup_{|\alpha|=k} ||\partial^{\alpha}u||_{L^q} \leq C^k{\lambda}^{k+d(\frac{1}{p}-\frac{1}{q})}||u||_{L^p}.
\end{equation}
Furthermore, if supp $\hat{u}\subset C(0, r_1\lambda, r_2\lambda)$, then 
\begin{equation}\label{bern2}
C^{-k}{\lambda}^k||u||_{L^p} \leq \sup_{|\alpha|=k}||\partial^{\alpha}u||_{L^p} \leq C^{k}{\lambda}^k||u||_{L^p}.
\end{equation} 
\end{lem} 

We will also make use of Osgood's Lemma in the proof of Theorem \ref{main}.  We refer the reader to \cite{CL} for a proof of the lemma.
\begin{lem}
Let $\rho$ be a measurable positive function, let $\gamma$ be a locally integrable positive function, and let $\mu$ be a continuous increasing function.  Assume that for some $\beta>0$, the function $\rho$ satisfies
\begin{equation*}
\rho(t) \leq \beta + \int_{t_0}^t \gamma(s)\mu(\rho(s)) \, ds.
\end{equation*}
Then $-\phi(\rho(t)) + \phi(\beta) \leq \int_{t_0}^t \gamma(s) \, ds$, where $\phi(x) = \int_x^1 \frac{1}{\mu(r)} \, dr$.
\end{lem}

In what follows, we set
\begin{align*}
	h_z(x) = 1/g_z(x) = (1 + |x-z|)^\alpha.
\end{align*}

The following lemma is a key technical ingredient in the proof of Theorem \ref{main}.
\begin{lem}\label{weight}
Assume $f$ belongs to $L^{\infty}(\R^d)$.  There exists a constant $C>0$, depending only on $\alpha$, such that for every $j\geq -1$,
\begin{equation}\label{weight1}
	\|g_z\Delta_j f \|_{L^{\infty}}
		\leq C\| g_z f \|_{L^{\infty}},
\end{equation}
and, for $1\leq k \leq d$,
\begin{equation}\label{weight2}
	\|g_z\partial_k\Delta_j f \|_{L^{\infty}}
		\leq C2^j\| g_z f \|_{L^{\infty}}.
\end{equation}
\end{lem}
\begin{proof}
We first prove (\ref{weight1}).  Observe that for any $\alpha>0$,
\begin{equation}\label{gh}
\begin{split}
&g_z(x)h_z(x-y) = \left( \frac{1+|x-y-z|}{1+|x-z|} \right)^{\alpha} \leq \left( \frac{1+|x-z| + |y|}{1+|x-z|} \right)^{\alpha} \leq (1+|y|)^{\alpha} .
\end{split}
\end{equation}
Assume $j\geq 0$.  By (\ref{gh}),
\begin{equation}\label{keyesthighfreq}
\begin{split}
&g_z(x)\int_{\R^d} \phi_j(y) f(x-y) \, dy = g_z(x)\int_{\R^d} \phi_j(y)h_z(x-y)g_z(x-y)f(x-y) \, dy \\
& \leq C\| g_z f \|_{L^{\infty}} \int_{\R^d} |\phi_j(y)|g_z(x)h_z(x-y) \, dy \leq C\| g_z f \|_{L^{\infty}}\int_{\R^d} |\phi_j(y)| (1+ |y|)^{\alpha} \, dy \\
&\qquad \leq C\| g_z f \|_{L^{\infty}}\int_{\R^d} |\phi(y)| (1+ |2^{-j}y|)^{\alpha} \, dy \leq C\| g_z f \|_{L^{\infty}}. 
\end{split}
\end{equation}
An identical argument can be applied for $j=-1$ with $\phi$ replaced by $\chi$.  This proves (\ref{weight1}).

To establish (\ref{weight2}), note that for $j\geq 0$,
\begin{equation}
\begin{split}
&g_z(x)\partial_k\Delta_j f (x) = g_z(x) \int_{\R^d} \partial_k\phi_j (y) f(x-y) \, dy \\
&\qquad = 2^{j(d+1)} \int_{\R^d} (\partial_k \phi) (2^jy) g_z(x)h_z(x-y) g_z(x-y)f(x-y) \, dy \\
&\qquad \leq C 2^j\|g_z f \|_{L^{\infty}} \int_{\R^d} (\partial_k \phi) (y) ( 1 + |2^{-j}y|)^{\alpha} \, dy \leq C 2^j\|g_z f \|_{L^{\infty}}.
\end{split}
\end{equation}
Again, an identical argument can be applied for $j=-1$ with $\phi$ replaced by $\chi$.  This completes the proof of Lemma \ref{weight}.
\end{proof}
In the proof of Theorem \ref{main}, the main restriction on the decay of $g_z$, and hence on the value of $\alpha$, comes from the pressure term.  To estimate this term, we will need the following theorem from \cite{WCU}, which gives a pointwise bound on the Riesz operators applied to sufficiently decaying functions with moment conditions.
\begin{thm}\label{Ward}
Assume $f$ is a function on $\R^d$ which satisfies:\\
\indent 1) $|f(x)| \leq C(1+|x|)^{-d-N+\epsilon}$,\\
\indent 2) $|D^{\alpha}f(x)|\leq C(1+|x|)^{-d -N -1 +\epsilon}$ for all $|\alpha|=1$, and\\
\indent 3) $\int x^{\beta} f(x) \, dx = 0$ for all $|\beta|<N$\\ 
for some fixed $\epsilon \in [0,1)$ and some positive integer $N$.
Then for each $i$, $1\leq i\leq d$, 
\begin{equation*}
|R_i f(x)| \leq C(1+|x|)^{-d-N +\epsilon + \delta}
\end{equation*}
for every $\delta$ satisfying $0 < \delta < 1-\epsilon$.
\end{thm}
We will apply Theorem \ref{Ward} to both $f(x)=\nabla\chi(x)$ and $f(x) = D^{\beta}\phi(x)$, where $\chi$ and $\phi$ are as above.  In particular, we will make use of the following corollary to Theorem \ref{Ward}, which is proved in \cite{Cozzi}.
\begin{lem}\label{decayestcor}
Assume $\eta$ belongs to $\mathcal{S}(\R^d)$.  Given $\epsilon>0$, there exists $C>0$ such that for each $i$, $j$ with $1\leq i$, $j\leq d$,
\begin{equation*}
|R_iR_j \nabla\eta(x)| \leq C(1+|x|)^{-d-1+\epsilon}.
\end{equation*}
Moreover, if $\psi$ belongs to $\mathcal{S}(\R^d)$ and has Fourier transform supported away from the origin, then given any $\gamma>0$, there exists $C>0$ such that for each $i$, $j$ with $1\leq i$, $j\leq d$,
\begin{equation*}
|R_iR_j \psi(x)| \leq C(1+|x|)^{-\gamma}.
\end{equation*}
\end{lem}    
We will also make frequent use of the following well-known lemma.  A proof of the lemma can be found in \cite{Taniuchi}.
\begin{lem}\label{BShighfreq}
Assume $u$ is a divergence-free vector field in $L^2_{loc}(\R^d)$ with curl $u=\omega$.  There exists $C>0$ such that for every $j\geq 0$, 
\begin{equation}
\| \Delta_j \nabla u \|_{L^{\infty}} \leq C\| \Delta_j \omega\|_{L^{\infty}}.
\end{equation}
\end{lem}
\section{Proof of Theorem \ref{main}}
The proof of Theorem \ref{main} is a modification of the proof of uniqueness of solutions to the Euler equations, as given in \cite{Vishik, TTY}.  The proof below deviates from those in \cite{Vishik, TTY} in that we must account for the weight $g_z$ when proving the necessary estimates.  To do this, we make use of Lemma \ref{weight} and Lemma \ref{decayestcor}.

Let $u_1$, $u_2$, $g_z$, and $T>0$ be as in the statement of Theorem \ref{main}.  Set $w = u_1 - u_2$.  For any $N \ge 0$,
\begin{equation}\label{linftytobesov}
\begin{split}
	&\| g_zw \|_{L^{\infty}}
		\leq \sum_{j=-1}^N \| g_z\Delta_j w \|_{L^{\infty}} + \sum_{j>N}
			\| g_z\Delta_j w \|_{L^{\infty}} \\
	&\qquad
	\leq CN \sup_{-1\leq j\leq N} \| g_z\Delta_j w \|_{L^{\infty}}
			+ \sum_{j>N} \| \Delta_j w \|_{L^{\infty}} \\
	&\qquad
	\leq CN \sup_{j} \| g_z\Delta_j w \|_{L^{\infty}}
			+ C \sum_{j>N} 2^{-j}(\| \Delta_j \omega_1\|_{L^{\infty}}
			+ \|\Delta_j \omega_2\|_{L^{\infty}}) \\
	&\qquad
	\leq CN \sup_{j} \| g_z\Delta_j w \|_{L^{\infty}}
		+ C \sum_{j>N} 2^{-j}(\| \Delta_j \omega_1\|_{L^{\infty}}
		+ \|\Delta_j \omega_2\|_{L^{\infty}}) \\
	&\qquad
		\leq CN \sup_{j} \| g_z\Delta_j w \|_{L^{\infty}} + CM2^{-N},
\end{split}
\end{equation}
where we used Bernstein's Lemma and Lemma \ref{BShighfreq} to get the third inequality.
To estimate $\| g_zw \|_{L^{\infty}}$, we establish an estimate for $\sup_{j} \| g_z\Delta_j w \|_{L^{\infty}}$ and only then judiciously choose a value of $N$ (see (\ref{N})).

We now estimate $\sup_{j} \| g_z\Delta_j w \|_{L^{\infty}}$, a task that will consume most of our effort.

Set $p=p_1 - p_2$.  Note that $w$ satisfies
\begin{equation*}
\partial_t w + w\cdot\nabla u_2 + u_1\cdot\nabla w = -\nabla p.
\end{equation*}
Applying the $\Delta_j$ operator to this equation gives
\begin{equation*}
\partial_t \Delta_j w + S_{j-2}u_1\cdot \nabla \Delta_jw  =  -Q_j(u_1,w) -\Delta_j(w\cdot\nabla u_2)-\nabla \Delta_j p,
\end{equation*}
where 
\begin{equation}
Q_j(u_1,w)=\Delta_j(u_1\cdot\nabla w) - S_{j-2}u_1\cdot \nabla \Delta_jw.
\end{equation}
We now multiply through by $g_z$ to get
\begin{equation*}
\partial_t (g_z\Delta_j w) + g_zS_{j-2}u_1\cdot \nabla \Delta_jw  =  -g_zQ_j(u_1,w) -g_z\Delta_j(w\cdot\nabla u_2)-g_z\nabla \Delta_j p.
\end{equation*}
Observe that, by the product rule,
\begin{equation*}
\begin{split}
&g_zS_{j-2}u_1\cdot \nabla \Delta_jw = S_{j-2}u_1\cdot g_z\nabla \Delta_jw\\
&\qquad = S_{j-2}u_1\cdot \nabla (g_z\Delta_jw) - (S_{j-2}u_1\cdot\nabla g_z)\Delta_j w.
\end{split}
\end{equation*}
This implies that
\begin{equation}\label{IthroughIV}
\partial_t (g_z\Delta_j w) + S_{j-2}u_1\cdot \nabla (g_z\Delta_jw) = I + II + III + IV,
\end{equation} 
where 
\begin{equation*}
\begin{split}
&I = (S_{j-2}u_1\cdot\nabla g_z)\Delta_j w\\
&II = -g_zQ_j(u_1,w)\\
&III = -g_z\Delta_j(w\cdot\nabla u_2)\\
&IV = -g_z\nabla \Delta_j p.
\end{split}
\end{equation*}

Fix $p_0 \ge 3$ to be specified later. We will estimate $I$-$IV$ for three cases: 1) middle frequencies, $3 \le j \le p_0$; 2) high frequencies, $j > p_0$; 3) low frequencies, $-1 \le j \le 2$, in that order.

So assume that $3 \leq j \leq p_0$.
We begin with $I$.  A simple calculation gives (\ref{gradgz}), which allows us to write
\begin{align}\label{I}
	\begin{split}
	\| I \|_{L^{\infty}}
		&= \|(S_{j-2}u_1\cdot \psi_z g_z)\Delta_j w\|_{L^{\infty}}
		\leq \alpha \| S_{j-2}u_1 \|_{L^{\infty}} \| g_z\Delta_j w \|_{L^{\infty}} \\
		&\leq CM\| g_z\Delta_j w \|_{L^{\infty}}.
	\end{split}
\end{align}
Here, we used $\|S_{j - 2} u_1\|_{L^\infty} \le C \|u_1\|_{L^\infty}$.

Next, we estimate $III$.  Expanding $III$ using Bony's paraproduct decomposition gives 
\begin{equation}\label{Bony}
\begin{split}
	&III
		= -g_z\Delta_j \sum_{|l-j|\leq 3} \Delta_l w \cdot\nabla S_{l-2} u_2 
			-g_z\Delta_j \sum_{|l-j|\leq 3} S_{l-2} w \cdot\nabla \Delta_l u_2 \\
		&\qquad
		-g_z\Delta_j \sum_{\stackrel{|l-l^{'}| \leq 1,}{\max\{l, l^{'}\} \geq j-3}}
			\Delta_l w \cdot\nabla \Delta_{l^{'}} u_2.
\end{split}
\end{equation}
The first term in (\ref{Bony}) can be estimated using Lemma \ref{weight} as follows:
\begin{equation}\label{first}
	\begin{split}
		&\largeabs{g_z\Delta_j \sum_{|l-j|\leq 3} \Delta_l w \cdot\nabla S_{l-2} u_2}
			\le C\sum_{|l-j|\leq 3} \|g_z\Delta_l w \cdot\nabla S_{l-2} u_2 \|_{L^{\infty}} \\
		&\qquad
		\le C\sum_{|l-j|\leq 3} \| S_{l-2} \nabla u_2\|_{L^{\infty}}\| g_z\Delta_l w\|_{L^{\infty}}
		\le CM(j-2)  \sup_q\| g_z\Delta_q w\|_{L^{\infty}}.
	\end{split}
\end{equation}
To obtain the last inequality above, we applied Bernstein's lemma and Lemma \ref{BShighfreq} to establish the series of estimates,
\begin{equation}\label{j-2}
\begin{split}
&\| S_{l-2} \nabla u_2\|_{L^{\infty}} \leq \| \Delta_{-1}\nabla u_2\|_{L^{\infty}} + \sum_{j=0}^{l-2} \| \Delta_j \nabla u_2\|_{L^{\infty}} \leq C\| u_2 \|_{L^{\infty}}  + C(l-2) \| \omega_2 \|_{L^{\infty}}.
\end{split}
\end{equation}
For the second term in (\ref{Bony}), another application of Lemma \ref{weight} and Lemma \ref{BShighfreq} gives
\begin{equation}\label{second}
	\begin{split}
		&\largeabs{g_z\Delta_j \sum_{|l-j|\leq 3} S_{l-2} w \cdot\nabla \Delta_l u_2}
			\le C\sum_{|l-j|\leq 3} \| g_z S_{l-2} w \cdot\nabla \Delta_l u_2\|_{L^{\infty}}\\
		&\qquad
			\le C\sum_{|l-j|\leq 3} \left( \| \Delta_l\nabla u_2 \|_{L^{\infty}}
				\sum_{m\leq l-2} \| g_z \Delta_m w\|_{L^{\infty}} \right)
 		\le CM (j-2) \sup_q \| g_z \Delta_q w\|_{L^{\infty}}. 
	\end{split}
\end{equation}
The third term in (\ref{Bony}) can be estimated using the divergence-free property of $w$, Lemma \ref{weight}, Bernstein's Lemma, and Lemma \ref{BShighfreq}.  We have the series of estimates 
\begin{equation}\label{third}
\begin{split}
	&\largeabs{g_z\Delta_j \sum_{\stackrel{|l-{l^{'}}|\leq 1,}{ \max\{l,l^{'}\}\geq j-3}}
			\Delta_l w \cdot\nabla \Delta_{l^{'}} u_2}
		= \largeabs{g_z\Delta_j \sum_{k=1}^d \sum_{\stackrel{|l-{l^{'}}|\leq 1,}
			{\max\{l,l^{'}\}\geq j-3}}
				\partial_k (\Delta_l w^k \Delta_{l^{'}} u_2)} \\
	&\le C\sum_{\stackrel{|l-{l^{'}}|\leq 1,} {\max\{l,l^{'}\}\geq j-3}}2^j 
			\| g_z \Delta_l w \Delta_{l^{'}} u_2 \|_{L^{\infty}}
	\le C\sum_{\stackrel{|l-{l^{'}}|\leq 1,} {\max\{l,{l^{'}}\}\geq j-3}}2^{j-{l^{'}}}\|
			\Delta_{l^{'}} \nabla u_2 \|_{L^{\infty}}\| g_z \Delta_l w\|_{L^{\infty}} \\
	&\le C\| \omega_2 \|_{L^{\infty}} \sup_q \| g_z \Delta_q w\|_{L^{\infty}}
	\le CM \sup_q \| g_z \Delta_q w\|_{L^{\infty}}. 
\end{split}
\end{equation}

Substituting (\ref{first}), (\ref{second}), and (\ref{third}) into (\ref{Bony}) yields the estimate
\begin{equation}
\| III \|_{L^{\infty}} \leq  CM (j-2) \sup_q \| g_z \Delta_q w\|_{L^{\infty}}.
\end{equation}
We now estimate $II$.  As in \cite{Vishik}, we write $Q_j(u_1,w) = \sum_{m=1}^4 Q^m_j(u_1,w)$, where
\begingroup
\allowdisplaybreaks
\begin{equation*}
\begin{split}
&Q^1_j(u_1,w) = \sum_{k=1}^d \Delta_j T_{\partial_kw}u^k_1\\
&Q^2_j(u_1,w) = - \sum_{k=1}^d [T_{u^k_1}\partial_k, \Delta_j]w\\
&Q^3_j(u_1,w) = \sum_{k=1}^d T_{u^k_1 - S_{j-2}u^k_1}\partial_k \Delta_j w\\
&Q^4_j(u_1,w) = \sum_{k=1}^d \{ \Delta_j R(u^k_1 \partial_k w) - R(S_{j-2}u^k_1, \Delta_j \partial_k w) \}. 
\end{split}
\end{equation*} 
\endgroup
To estimate $g_zQ^1_j(u_1,w)$, we use the divergence-free property of $u_1$, Lemma \ref{weight}, Bernstein's Lemma, and Lemma \ref{BShighfreq} to write
\begin{equation*}
\begin{split}
	&\abs{g_zQ^1_j(u_1,w)}
		= \largeabs{g_z\sum_{k=1}^d \Delta_j\partial_k \sum_{|l-j|\leq 3} S_{l-2} w \Delta_l u^k_1}
		\le C  \sum_{|l-j|\leq 3} 2^j\| g_z S_{l-2} w \Delta_l u_1 \|_{L^{\infty}} \\
		&\qquad
		\le  C\sum_{|l-j|\leq 3} 2^{j-l} (l-2)\sup_q \| g_z\Delta_q w \|_{L^{\infty}}\|
			\Delta_l \nabla u_1 \|_{L^{\infty}} \\
		&\qquad
		\le C\|\omega_1 \|_{L^{\infty}} (j-2)\sup_q \| g_z\Delta_q w \|_{L^{\infty}}
		\le CM (j-2)\sup_q \| g_z\Delta_q w \|_{L^{\infty}}. 
\end{split}
\end{equation*}

Next, we estimate $g_z Q^3_j(u_1,w)$.  By properties of the Littlewood-Paley operators, there exists $M_0>0$ such that
\begin{equation*}
\begin{split}
	&Q^3_j(u_1,w)
		= \sum_{k=1}^d\sum_{l=1}^{\infty} S_{l-2}(u^k_1 - S_{j-2}u^k_1)
			\Delta_l \partial_k\Delta_j w \\
		&\quad
		= \sum_{k=1}^d\sum_{|l-j|\leq 3} S_{l-2}
			\left(\sum_{m=j-1}^{\infty}\Delta_mu^k_1\right)
			\Delta_l \partial_k\Delta_j w
		= \sum_{k=1}^d \sum_{|l-j|\leq 3} \left( \sum_{|m-j| \leq M_0} \Delta_mu^k_1\right)
			\Delta_l\partial_k\Delta_j w. 
\end{split}
\end{equation*}
Hence,
\begin{equation}\label{R3}
\begin{split}
	&\abs{g_zQ^3_j(u_1,w)}
		\le C\sum_{k=1}^d \sum_{|l-j|\leq 3} \left( \sum_{|m-j| \leq M_0}2^{-m}\|\Delta_m
			\nabla u_1 \|_{L^{\infty}} \right)\|g_z\partial_k\Delta_l\Delta_j w\|_{L^{\infty}} \\
		&\qquad \le C\|\omega_1\|_{L^{\infty}}\sum_{k=1}^d \sum_{|l-j|\leq 3} 2^{-j}\|
			g_z\partial_k\Delta_l\Delta_j w\|_{L^{\infty}},
\end{split}
\end{equation}
where Bernstein's Lemma was used to get the first inequality and Lemma \ref{BShighfreq} was applied to get the second inequality.  Now note that by Lemma \ref{weight},  
\begin{equation*}
\begin{split}
&\|g_z\partial_k\Delta_l\Delta_j w\|_{L^{\infty}} \leq C2^{l}\| g_z\Delta_j w\|_{L^{\infty}}.\\ 
\end{split}
\end{equation*}
Inserting this estimate into (\ref{R3}) gives 
\begin{equation*}
	\abs{g_zQ^3_j(u_1,w)}
		\le CM \sup_q \| g_z\Delta_q w\|_{L^{\infty}}.
\end{equation*}

To estimate $g_zQ^2_j(u_1,w)$, we observe that there exists $M_0>0$ such that 
\begin{equation*}
	Q^2_j(u_1,w)
		=  -\sum_{k=1}^d \sum_{|l-j|\leq M_0} [S_{l-2}u^k_1 , \Delta_j] \partial_k\Delta_{l} w.
\end{equation*}
We express the commutator as a convolution integral and utilize the divergence-free property of $u_1$ to write,
\begin{equation*}
\begin{split}
	&Q^2_j(u_1,w)(x)
		= -
		\sum_{|l-j|\leq M_0} \sum_{k=1}^d
			\int_{\R^d} \phi_j(y) (S_{l-2}u^k_1(x) - S_{l-2}u^k_1(x-y))
				\partial_k\Delta_l w (x-y) \, dy \\
		&\qquad = 
		-
		\sum_{|l-j|\leq M_0} \sum_{k=1}^d
			\int_{\R^d} \partial_k\phi_j(y) (S_{l-2}u^k_1(x) - S_{l-2}u^k_1(x-y))
				\Delta_l w (x-y) \, dy.
\end{split}
\end{equation*}
Hence,
\begin{equation*}
\begin{split}
	&\abs{g_zQ^2_j(u_1,w)(x)} \\
		&\quad \le g_z(x) \hspace{-0.5em} \sum_{|l-j|\leq M_0} 2^{j(d+1)}
			\int_{\R^d} |\nabla\phi(2^jy)| \|S_{l-2}\nabla u_1 \|_{L^{\infty}}|y|
				g_z(x-y)h_z(x-y) |\Delta_l w (x-y)| \, dy \\
		&\quad \le \sum_{|l-j|\leq M_0} 2^{j(d+1)} \| g_z\Delta_l w \|_{L^{\infty}}
			\|S_{l-2}\nabla u_1 \|_{L^{\infty}}
			\int_{\R^d} |\nabla\phi(2^jy)| |y| g_z(x)h_z(x-y) \, dy \\
		&\quad \le \sum_{|l-j|\leq M_0} 2^j \| g_z\Delta_l w \|_{L^{\infty}}
			\|S_{l-2}\nabla u_1 \|_{L^{\infty}}
			\int_{\R^d} |\nabla\phi(y)||2^{-j}y|(1+|2^{-j}y|)^{\alpha} \, dy \\
		&\quad \le C M(j-2) \sup_q \| g_z \Delta_q w\|_{L^{\infty}},
\end{split}
\end{equation*}
where we applied (\ref{gh}) and a change of variables to get the second-to-last inequality and (\ref{j-2}) to get the last inequality above.

For $g_zQ^4_j(u_1,w)$, we again follow \cite{Vishik} and write $Q^4_j(u_1,w)$ as the sum of two terms:
\begin{equation*}
Q^4_j(u_1,w) = Q^{4,1}_j(u_1,w) + Q^{4,2}_j(u_1,w),
\end{equation*}
where 
\begin{equation*}
\begin{split}
&Q^{4,1}_j(u_1,w) = \sum_{k=1}^d \Delta_j\partial_kR(u^k_1 - S_{j-2}u^k_1, w), \text{ and }\\
&Q^{4,2}_j(u_1,w) = \sum_{k=1}^d \{ \Delta_j R(S_{j-2}u^k_1, \partial_k w) - R(S_{j-2}u^k_1, \Delta_j \partial_k w) \}.
\end{split}
\end{equation*}
By properties of Littlewood Paley operators and Lemma \ref{weight},
\begin{equation*}
\begin{split}
	&\abs{g_zQ^{4,1}_j(u_1,w)}
		= \largeabs{g_z \sum_{k=1}^d \Delta_j\partial_k \sum_{\stackrel{|l-l^{'}|\leq 1,}
			{\max\{l,l^{'}\}\geq j-3}} \Delta_{l}(u^k_1 - S_{j-2}u^k_1) \Delta_{l^{'}} w} \\
	&\qquad \le
		C\sum_{k=1}^d\sum_{\stackrel{|l-l^{'}|\leq 1,} {\max\{l,l^{'}\}\geq j-3}}2^j
			\|g_z \Delta_{l}(u^k_1 - S_{j-2}u^k_1)\Delta_{l^{'}}w\|_{L^{\infty}} \\
	&\le C \sup_{q} \|g_z\Delta_{q} w\|_{L^{\infty}}
		\sum_{l\geq j-M_0}2^{j-l}\|\Delta_{l} \nabla u_1\|_{L^{\infty}}
	\le CM \sup_{q} \|g_z\Delta_{q} w\|_{L^{\infty}},
\end{split}
\end{equation*}
where Lemma \ref{BShighfreq} was again used to get the last inequality.  

The strategy for estimating $Q^{4,2}_j(u_1,w)$ is similar to that used to bound $g_zQ^2_j(u_1,w)$.  By properties of Littlewood Paley operators and the divergence-free property of $u_1$,
\begin{equation*}
\begin{split}
	&Q^{4,2}_j(u_1,w)
		= g_z\sum_{k=1}^d \sum_{l=j-M_0}^{j-1} \sum_{l^{'}=l-1}^{l+1}
			[\Delta_j, \Delta_lS_{j-2}u^k_1]\Delta_{l^{'}}\partial_k w \\
	&= \sum_{k=1}^d \sum_{l=j-M_0}^{j-1} \sum_{l^{'}=l-1}^{l+1} 2^{jd}
		\int_{\R^d} \phi(2^jy)[\Delta_lS_{j-2}u^k_1(x) - \Delta_lS_{j-2}u^k_1(x-y)]
			\Delta_{l^{'}}\partial_k w(x-y) \, dy \\
	&= \sum_{k=1}^d \sum_{l=j-M_0}^{j-1}
		\sum_{l^{'}=l-1}^{l+1} 2^{j(d+1)}\int_{\R^d}
		(\partial_k\phi)(2^jy)[\Delta_lS_{j-2}u^k_1(x) - \Delta_lS_{j-2}
		u^k_1(x-y)] \Delta_{l^{'}} w(x-y) \, dy.
\end{split}
\end{equation*}
Hence,
\begin{equation*}
\begin{split}
	&\abs{g_zQ^{4,2}_j(u_1,w)} \\
	&\qquad \le \sum_{l=j-M_0}^{j-1} \sum_{l^{'}=l-1}^{l+1} 2^{j(d+1)}
		\int_{\R^d} |\nabla\phi(2^jy)| \| \Delta_l\nabla S_{j-2}u_1 \|_{L^{\infty}}
			|y| g_z(x)h_z(x-y)\| g_z\Delta_{l^{'}} w\|_{L^{\infty}} \, dy \\
	&\qquad \le \sum_{l=j-M_0}^{j-1} \sum_{l^{'}=l-1}^{l+1} 2^{j(d+1)} \|
		\Delta_l\nabla u_1 \|_{L^{\infty}} \| g_z\Delta_{l^{'}} w\|_{L^{\infty}}
		\int_{\R^d} |\nabla\phi(2^jy)|  |y| (1+ |y|)^{\alpha} \, dy \\
	&\qquad \le C \sum_{l=j-M_0}^{j-1} \sum_{l^{'}=l-1}^{l+1} \| \omega_1 \|_{L^{\infty}} \|
		g_z\Delta_{l^{'}} w\|_{L^{\infty}} \int_{\R^d} |\nabla\phi(y)|
		|y| (1+ |2^{-j}y|)^{\alpha} \, dy \\
	&\qquad \le CM \sup_q \| g_z\Delta_{q} w\|_{L^{\infty}},
\end{split}
\end{equation*}
where we again used (\ref{gh}) and Lemma \ref{BShighfreq}.

We conclude that
\begin{align*}
	\norm{II}_{L^\iny}
		\le C M(j-2) \sup_q \| g_z \Delta_q w\|_{L^{\infty}}.
\end{align*}

We now estimate $IV$.  We follow arguments from \cite{TTY}.  Set $\tilde{\phi}_j = \phi_{j-1} + \phi_j + \phi_{j+1}$, so that $\tilde{\Delta}_j = \Delta_{j-1} + \Delta_j + \Delta_{j+1}$.  As in \cite{TTY}, write
\begin{equation}
\begin{split}
&IV = -g_z\nabla \Delta_j p = -g_z\nabla(-\Delta)^{-1}\sum_{k,l=1}^d \partial_k \tilde{\Delta}_j\partial_l\Delta_j (w^lu_1^k)\\
&\qquad   -g_z\nabla(-\Delta)^{-1}\sum_{k,l=-1}^d \partial_l \tilde{\Delta}_j\partial_k\Delta_j (w^k u_2^l)  =: V + VI.
\end{split}
\end{equation}
For $V$, observe that
\begin{equation*}
	\begin{split}
		&V
		= g_z(x)\sum_{k,l=1}^d \int_{\R^d} \nabla(-\Delta)^{-1} \partial_k\tilde{\phi}_j (y)
				\partial_l\Delta_j (w^lu_1^k)(x-y) \, dy,
\end{split}
\end{equation*}
so that
\begin{equation}\label{pressureV}
	\begin{split}
		&\abs{V}
		\le \sum_{k,l=1}^d \|g_z\partial_l\Delta_j (w^lu^k_1)
				\|_{L^{\infty}}\int_{\R^d} |\nabla(-\Delta)^{-1}
				\partial_k\tilde{\phi}_j (y)|g_z(x)h_z(x-y)  \, dy \\
		&\qquad
		\le \sum_{k,l=1}^d \|g_z\partial_l\Delta_j (w^lu^k_1)
				\|_{L^{\infty}}\int_{\R^d} |\nabla(-\Delta)^{-1}
				\partial_k\tilde{\phi}_j (y)|(1+ |y|)^{\alpha} \, dy, \\
\end{split}
\end{equation}
where we used (\ref{gh}) to get the last inequality.  The operator $\nabla(-\Delta)^{-1} \partial_k$ can be expressed as a product of Riesz operators, which commute with dilations. Thus, a change of variables gives 
\begin{equation}\label{Vstep2}
\begin{split}
&\abs{V} \leq    \sum_{k,l=1}^d \|g_z\partial_l\Delta_j (w^lu^k_1) \|_{L^{\infty}}\int_{\R^d} |\nabla(-\Delta)^{-1} \partial_k\tilde{\phi} (y)|(1+ |2^{-j}y|)^{\alpha}  \, dy \\
&\qquad \leq   \sum_{k,l=1}^d \|g_z\partial_l\Delta_j (w^lu^k_1) \|_{L^{\infty}} \int_{\R^d} |\nabla(-\Delta)^{-1} \partial_k\tilde{\phi} (y)|(1+ |y|)^{\alpha}  \, dy.
\end{split}
\end{equation}
Observe that, since $\tilde{\phi}$ has Fourier transform supported away from the origin, Lemma \ref{decayestcor} can be applied.  Thus the integral in (\ref{Vstep2}) is finite, and    
\begin{equation}
\abs{V} \leq C\sum_{k,l=1}^d\|g_z\partial_l\Delta_j (w^lu_1^k) \|_{L^{\infty}}.
\end{equation}

We estimate $VI$ in exactly the same way as $V$, and we conclude that
\begin{equation}\label{pressureVI}
\abs{VI} \leq C\sum_{k,l=1}^d\|g_z\partial_k\Delta_j (w^ku_2^l) \|_{L^{\infty}}.
\end{equation}

We can now complete the estimate for $V$ and $VI$ by applying an argument identical to that used to estimate $III$.  This gives
\begin{equation*}
	\norm{IV}_{L^\iny}
		\le \norm{V}_{L^\iny} + \norm{VI}_{L^\iny}
		\le CM(j-2)\sup_q \|g_z\Delta_q w \|_{L^{\infty}}.
\end{equation*} 

Viewing (\ref{IthroughIV}) as a transport equation ($g_z\Delta_j w$ transported by $S_{j - 2} u_1$), integrating in time, and applying the estimates for $I$, $II$, $III$, and $IV$ above, we conclude that for $3 \leq j\leq p_0$,
\begin{equation*}
 \| g_z\Delta_j w(t) \|_{L^{\infty}} \leq   \| g_z\Delta_j w^0 \|_{L^{\infty}} + CM\int_0^t \left( (j-2) \sup_q  \| g_z\Delta_q w(s) \|_{L^{\infty}}   \right)  \, ds,
\end{equation*} 
so that
\begin{equation}\label{midfreq}
 \sup_{3 \leq j \leq p_0}\| g_z\Delta_j w(t) \|_{L^{\infty}} \leq   \sup_{3 \leq j \leq p_0}\| g_z\Delta_j w^0 \|_{L^{\infty}} + CM\int_0^t \left(  p_0\sup_q  \| g_z\Delta_q w(s) \|_{L^{\infty}}   \right)  \, ds.
\end{equation} 

For the case $j>p_0$, note that by Bernstein's lemma and Lemma \ref{BShighfreq},
\begin{equation}\label{highfreq}
\sup_{j\geq p_0} \| g_z\Delta_j w(t) \|_{L^{\infty}} \leq \sup_{j\geq p_0} 2^{-j}\| \Delta_j \nabla w(t) \|_{L^{\infty}} \leq CM 2^{-p_0}.
\end{equation}

Now assume $-1\leq j \leq 2$.  In this case, $\| I \|_{L^{\infty}}$ can be handled exactly as in (\ref{I}).  For $III$, the divergence-free assumption on $u_2$ and Lemma \ref{weight} imply that 
\begin{equation*}
\begin{split}
&\abs{III} = \abs{g_z\Delta_j\nabla\cdot(wu_2)} \leq C2^j \| g_zwu_2 \|_{L^{\infty}} \leq CM \| g_zw \|_{L^{\infty}}.
\end{split}
\end{equation*}  
For $IV$, we write
\begin{equation}\label{pressurelowfreq3}
g_z\nabla\Delta_j p = g_z\sum_{i,k=1}^d R_i R_k \nabla\Delta_j (w^iu_1^k + u_2^iw^k).
\end{equation}

We focus on $g_z R_i R_k \nabla\Delta_j (w^iu_1^k)$.  Note that for $0\leq j \leq 2$,
\begin{equation*}
\begin{split}
	&g_z R_i R_k \nabla\Delta_j (w^iu_1^k)
		= \int_{\R^d} (R_iR_k\nabla\phi_j(y)) g_z(x) h_z(x-y) u_1^k (x-y) (g_zw^i)(x-y) \, dy,
\end{split}
\end{equation*}
so
\begin{equation}\label{pressurelowfreq}
\begin{split}
	&\abs{g_z R_i R_k \nabla\Delta_j (w^iu_1^k)}
		\le C\| u_1 \|_{L^{\infty}}\| g_zw \|_{L^{\infty}}
			\int_{\R^d} |R_iR_k\nabla\phi_j(y)| (1+ |y|)^{\alpha} \, dy \\
		&\qquad \le CM\| g_zw \|_{L^{\infty}} \int_{\R^d} |R_iR_k\nabla\phi(y)|
			(1+ |2^{-j}y|)^{\alpha} \, dy,
\end{split}
\end{equation}
where we used (\ref{gh}) to get the first inequality.  Since $\nabla\phi$ has Fourier support away from the origin, Lemma \ref{decayestcor} can be applied to conclude that the integral in (\ref{pressurelowfreq}) is finite.  Thus,
\begin{equation*}
\abs{g_z R_i R_k \nabla\Delta_j (w^iu_1^k)} \leq CM\| g_zw \|_{L^{\infty}}.
\end{equation*}

For $j=-1$, we must estimate
\begin{equation}\label{pressurelowfreq1}
g_z R_i R_k \nabla\Delta_{-1} (w^iu_1^k) = g_z(x) \int_{\R^d} R_iR_k\nabla\chi(y) (w^iu_1^k) (x-y) \, dy.
\end{equation} 
This term is the only term which places a substantial restriction on the growth of $h_z$ (or decay of $g_z$) at infinity.  Indeed, estimates for all other terms in the proof of Theorem \ref{main} will hold with $h_z(x) = (1+|x-z|)^{\alpha}$ for any $\alpha>0$.  It is in the estimate for (\ref{pressurelowfreq1}) where we utilize the membership of $u_1$ and $u_2$ in $L^p(\R^d)$ to optimize the value of $\alpha$.

Let $q$ satisfy $1\slash p + 1 \slash q =1$. Then   
\begin{equation*}
\begin{split}
	&g_z R_i R_k \nabla\Delta_{-1} (w^iu_1^k)
	= g_z(x) \int_{\R^d} R_iR_k\nabla\chi(y) h_z(x-y) u_1^k (x-y) (g_zw^i)(x-y) \, dy, \\
\end{split}
\end{equation*}
and hence,
\begin{equation}\label{pressurelowfreq2}
\begin{split}
	&\abs{g_z R_i R_k \nabla\Delta_{-1} (w^iu_1^k)}
		\le \|R_iR_k\nabla\chi(\cdot) g_z(x)h_z(x-\cdot)\|_{L^q} \|u_1^k  (g_zw^i) \|_{L^p} \\
		&\le \|R_iR_k\nabla\chi(\cdot) (1+|\cdot|)^{\alpha}\|_{L^q} \|u_1 \|_{L^p}
		\|g_zw \|_{L^{\infty}},
\end{split}
\end{equation}
where we used (\ref{gh}) to get the last inequality.

Given $p$, we must determine values of $\alpha$ for which $\|R_iR_k\nabla\chi(\cdot) (1+|\cdot|)^{\alpha}\|_{L^q}$ is finite.  In view of Lemma \ref{decayestcor}, a short calculation shows that $\|R_iR_k\nabla\chi(\cdot) (1+|\cdot|)^{\alpha}\|_{L^q}$ is finite if
\begin{equation*}
\alpha < 1+ d\slash p.
\end{equation*}
Therefore, assuming $\alpha < 1+ d\slash p$, it follows that
\begin{equation*}
	\abs{g_z R_i R_k \nabla\Delta_{-1} (w^iu_1^k)}
		\le CM \|g_zw \|_{L^{\infty}}.  
\end{equation*}

Of course, arguments identical to those above imply that
\begin{equation*}
\begin{split}
	&\abs{g_z R_i R_k \nabla\Delta_{-1}(u_2^iw^k)}
		\le CM \|g_zw \|_{L^{\infty}}, \text{ and } \\
	&\abs{g_z R_i R_k \nabla\Delta_{j}(u_2^iw^k)}
		\le CM \|g_zw\|_{L^{\infty}} 
\end{split}
\end{equation*}  
for $0\leq j \leq 2$.  Inserting these estimates into (\ref{pressurelowfreq3}) gives the following estimate for $-1\leq j \leq 2$:
\begin{equation*}
	\abs{g_z\nabla\Delta_{j}p}
		\le CM\|g_zw \|_{L^{\infty}}.
\end{equation*}

We now turn our attention to $II$.  Instead of considering the commutator, we estimate each term in the difference individually.  For the first term in the difference, we observe that, by Lemma \ref{weight},
\begin{equation*}
\begin{split}
	&\abs{g_zS_{j-2}u_1\cdot\nabla\Delta_j w} 
		\le \| S_{j-2}u_1g_z \nabla \Delta_j w \|_{L^{\infty}} \\
	&\qquad \le \| S_{j-2}u_1 \|_{L^{\infty}} \| g_z \nabla \Delta_j w \|_{L^{\infty}}
	\le CM 2^j \| g_z  \Delta_j w \|_{L^{\infty}}
\end{split}
\end{equation*}
for $-1\leq j \leq 2$.  For the second term in the difference, note that by the divergence-free property of $u_1$ and Lemma \ref{weight},
\begin{equation*}
\begin{split}
	&\abs{g_z\Delta_j (u_1\cdot\nabla w)}
		\le C2^j \| u_1g_zw \|_{L^{\infty}}\leq CM \|g_zw \|_{L^{\infty}}. 
\end{split}
\end{equation*}

We also recall the following estimate, which is derived in (\ref{linftytobesov}) and which holds for any $p_0\geq 0$:
\begin{equation}\label{inftytobesov2}
\| g_zw \|_{L^{\infty}} \leq Cp_0 \sup_{q} \| g_z\Delta_q w \|_{L^{\infty}} + CM2^{-p_0}.
\end{equation}

Integrating (\ref{IthroughIV}) in time, applying the above estimates for $I$, $II$, $III$, and $IV$ when $-1\leq j\leq 2$, applying (\ref{inftytobesov2}), and taking the supremum over $-1\leq j\leq 2$ gives   
\begin{equation*}
\sup_{-1\leq j \leq 2} \| g_z\Delta_j w(t) \|_{L^{\infty}} \leq \sup_{-1\leq j \leq 2} \| g_z\Delta_j w^0 \|_{L^{\infty}} + CM \int_0^t ( p_0 \sup_{q} \| g_z\Delta_q w(s) \|_{L^{\infty}} + 2^{-p_0} ) \, ds.
\end{equation*}
Combining this estimate with (\ref{midfreq}) and (\ref{highfreq}) yields
\begin{align}\label{preOsgood}
	\begin{split}
	\sup_q &\| g_z\Delta_q w(t) \|_{L^{\infty}} \\
		&\le \sup_q \| g_z\Delta_q w^0 \|_{L^{\infty}}
			+ CM \int_0^t \left(  p_0\sup_q  \| g_z\Delta_q w(s) \|_{L^{\infty}}
				+ 2^{-p_0} \right) \, ds \\
		&\le C \epsilon_0
			+ CM \int_0^t \left(  p_0\sup_q  \| g_z\Delta_q w(s) \|_{L^{\infty}}
				+ 2^{-p_0} \right) \, ds,
	\end{split}
\end{align}
where
\begin{align*}
	\epsilon_0 := \norm{g_z w^0}_{L^\iny},
\end{align*}
and where in the last inequality we applied Lemma \ref{weight}.

At this point, we could apply Gronwall's estimate to obtain
\begin{align*}
	\sup_q \| g_z\Delta_q w(t) \|_{L^{\infty}}
		\le \pr{C\epsilon_0 + CM 2^{-p_0} t} e^{CM p_0 t}
\end{align*}
and choose $p_0$ in a manner so as to minimize the right-hand side. This will not, however, yield a sufficiently tight estimate. Instead,
we will choose $p_0$ in (\ref{preOsgood}) to vary with $t$ (this is valid, because the constant, $C$, in (\ref{preOsgood}) does not depend on $p_0$) in a way that allows us to apply Osgood's Lemma.

Toward this end, first observe that it follows from Lemma \ref{weight} that for all $t\geq 0$,
\begin{equation*}
	\delta(t)
 		:= \frac{1}{C_0Mt}\int_0^t \sup_k \| g_z\Delta_k w(s) \|_{L^{\infty}} \, ds
		\le 1,
 \end{equation*}
where $C_0$ is the constant in (\ref{weight1}).
Now choose 
\begin{equation*}
p_0=3-\log_2 \delta(t).
\end{equation*}
Note that $p_0 \ge 3$ as required.

Substituting our chosen value for $p_0$ into (\ref{preOsgood}) yields
\begin{align}\label{preOsgood1}
	\begin{split}
		\sup_q \| &g_z\Delta_q w(t) \|_{L^{\infty}}
			\le C_0 \epsilon_0
				+ CMt\delta(t) + CM^2t(3-\log_2 \delta(t))\delta(t) \\
			&\le C_0 \epsilon_0
				+ CM^2t \delta(t) (3 - \log_2 \delta(t))
			\le C_0 \epsilon_0
				+ CM^2t \delta(t) (3 - \log \delta(t)).
	\end{split}
\end{align}

Setting $t=s$ in (\ref{preOsgood1}) and integrating both sides from $0$ to $t$, then dividing both sides by $C_0 Mt$ gives
\begin{align*}
	\delta(t)
		&\le \frac{C_0 \epsilon_0}{C_0 M}
			+ CM \frac{1}{t} \int_0^t s \delta(s) (3 - \log \delta(s)) \, ds
		\le \frac{\epsilon_0}{M}
			+ C M \int_0^t \delta(s) (3 - \log \delta(s)) \, ds.
\end{align*}

Since $x \mapsto x (3 - \log x)$ is increasing on $(0, 1]$ and we know that $\delta(s) \in (0, 1]$, we can apply Osgood's Lemma to obtain
\begin{align*}
	-\log (3 - \log \delta(t)) + \log \pr{3 - \log \pr{\frac{\epsilon_0}{M}}}
		\le C M t.
\end{align*}
Taking the exponential of both sides twice,
we conclude that
\begin{align*} 
	\delta(t)
		\le \overline{\delta}(t)
		:= e^3 \pr{\frac{\epsilon_0}{M}}^{e^{-C Mt}}.
\end{align*}
\ScratchWork{
	Let $B_0 = \epsilon_0/M$. Then exponentiating once gives
	\begin{align*}
		\frac{3 - \log B_0}{3 - \log \delta(t)}
			\le e^{CMt}.
	\end{align*}
	Noting the numerator and denominator of the left hand side along with the right hand side
	are each positive, we have
	\begin{align*}
		3 - \log \delta(t)
			\ge (3 - \log B_0) e^{- CMt}
	\end{align*}
	so
	\begin{align*}
		\log \delta(t)
			\le 3 - (3 - \log B_0) e^{- CMt}
	\end{align*}
	so
	\begin{align*}
		\delta(t)
			\le e^3 e^{- (3 - \log B_0) e^{- CMt}}
			= e^3 e^{(\log B_0 - 3) e^{- CMt}}
			\le e^3 e^{\log B_0 e^{- CMt}}
			= e^3 B_0^{e^{C M t}}.
	\end{align*}
} 

Now increase the bound in (\ref{preOsgood1}) to
\begin{align}\label{preOsgood2}
	\sup_q \| &g_z\Delta_q w(t) \|_{L^{\infty}}
		\le C_0 \epsilon_0 + CM^2t \delta(t) (4 - \log \delta(t))
\end{align}
and note that $\overline{\delta}(t) \le e^3$.
Then since $x \mapsto x (4 - \log x)$ is increasing for $x \in (0, e^3]$, we can replace $\delta(t)$ in (\ref{preOsgood2}) with $\overline{\delta}(t)$. This gives,
\begin{align}\label{B0Bound}
	\begin{split}
		\sup_q \| g_z&\Delta_q w(t) \|_{L^{\infty}}
			\le C_0 \epsilon_0 + CM^2t e^3 \pr{\frac{\epsilon_0}{M}}^{e^{-CMt}}
				\pr{4 - \log \pr{e^3 \pr{\frac{\epsilon_0}{M}}^{e^{-CMt}}}} \\
			&= C_0 \epsilon_0 + C M^2t e^3 \pr{\frac{\epsilon_0}{M}}^{e^{-CMt}}
				\pr{1 - e^{-CMt} \log \pr{\frac{\epsilon_0}{M}}}
				:= a.
	\end{split}
\end{align}

By (\ref{linftytobesov}),
\begin{align*}
	\norm{g_z w}_{L^\iny}
		\le C (N a + M 2^{-N}).
\end{align*}
Define $F(N) := N a + M 2^{-N}$. This has a minimum when
\begin{align}\label{N}
	N = N_0 := \log_2 (M \log 2/a),	
\end{align}
where we treat $N_0$ as though it were an integer. At this (near) minimum, assuming that $0 < a < M \log 2$, which holds for all sufficiently small $\norm{g_z w^0}_{L^\iny}$, we have
\begin{align*}
	F(N_0)
		&= a \log_2 \frac{M \log 2}{a} + \frac{Ma}{M \log 2}
		= M \frac{a}{M \log 2} - M \log 2 \frac{a}{M \log 2} \log_2 \frac{a}{M \log 2} \\
		&= M(b - (\log 2) b \log_2 b)
		= M b(1 - \log b),
\end{align*}
where
$
		b := a/(M \log 2) < 1.
$
(Note that if $a > M \log 2$, we can do no better than setting $N = 0$, which gives the useless bound, $\norm{g_z w}_{L^\iny} \le C M$.)

Hence, when $a < M \log 2$, we have
\begin{align*}
	\norm{g_z w}_{L^\iny}
		\le C M b(1 - \log b),
\end{align*}
which is the bound stated in Theorem \ref{main}.

\ScratchWork{
For fixed $b \in (0, 1)$, define the strictly increasing function, $F_b \colon [0, \iny) \mapsto [0, \iny)$ by
\begin{align*}
	F_b(t) = t b^{e^{-CMt}} (1 - e^{-CMt} \log b).
\end{align*}
Then
\begin{align*}
	a \le M \log 2
		&\iff C_1 \epsilon_0 + CM^2 e^3 F_{\pr{\frac{\epsilon_0}{M}}}(t)
			\le M \log 2 \\
		&\iff F_{\pr{\frac{\epsilon_0}{M}}}(t)
			\le C \brac{\frac{\log 2}{M} - \frac{C_1 \epsilon_0}{M^2}} \\
		&\iff \epsilon_0 \le \frac{M \log2}{C_1}
				\text{ and }
			t \le F_{\pr{\frac{\epsilon_0}{M}}}^{-1}
				\pr{C \brac{\frac{\log 2}{M} - \frac{C_1 \epsilon_0}{M^2}}}.
\end{align*}
} 



\begin{acknowledgement}
	The problem this paper addresses was suggested to the authors by Helena Nussenzveig
	Lopes and Milton Lopes Filho in the context of Serfati solutions to the 2D Euler
	equations. The National Science Foundation provided support for EC through grant
	DMS-1160704 and support for JPK through grant DMS-1212141.
\end{acknowledgement}

\end{document}